\documentclass[a4paper,english]{article}

\usepackage[utf8]{inputenc}
\usepackage[T1]{fontenc}
\usepackage{babel}
\usepackage[babel]{csquotes}

\usepackage{amsmath}
\usepackage{amsfonts}
\usepackage{amsthm}
\usepackage{amssymb}
\usepackage{url}

\usepackage{enumitem} 
\setenumerate[0]{label=(\roman*)}   

\newtheorem{theorem}{Theorem}[section]
\newtheorem{proposition}[theorem]{Proposition}
\newtheorem{lemma}[theorem]{Lemma}

\theoremstyle{definition}
\newtheorem{definition}[theorem]{Definition}
\newtheorem{example}{Example}
\newtheorem{remark}{Remark}

\DeclareMathOperator{\Ad}{Ad}

\DeclareMathOperator{\End}{End}
\DeclareMathOperator{\id}{id}

\DeclareMathOperator{\Grass}{Grass}
\DeclareMathOperator{\Supp}{Supp}
\DeclareMathOperator{\Stab}{Stab}

\newcommand{\bracket}[1]{\langle#1\rangle}
\newcommand{\ND}[1]{N(#1,\delta)}
\newcommand{\NR}[1]{N(#1,\rho)}

\newcommand{\eps}{\varepsilon}
\newcommand{\GL}{\mathrm{GL}}
\newcommand{\SL}{\mathrm{SL}}

\newcommand{\C}{\mathbb{C}}
\newcommand{\R}{\mathbb{R}}
\newcommand{\Z}{\mathbb{Z}}
\newcommand{\Q}{\mathbb{Q}}
\newcommand{\N}{\mathbb{N}}

\newcommand{\ba}{\mathbf{a}}
\newcommand{\bb}{\mathbf{b}}

\newcommand{\g}{\mathfrak{g}}
\newcommand{\h}{\mathfrak{h}}

\newcommand{\s}{\mathfrak{s}}
\newcommand{\rf}{\mathfrak{r}}
\newcommand{\nf}{\mathfrak{n}}

\newcommand{\cG}{\mathcal{G}}
\newcommand{\cF}{\mathcal{F}}

\newcommand{\cP}{\mathcal{P}}

\newcommand{\norm}[1]{\lVert#1\rVert}
\renewcommand{\phi}{\varphi} 

\newcommand{\dd}{\,\mathrm{d}}
\DeclareMathOperator{\Mat}{Mat}
\DeclareMathOperator{\Derive}{D}
\DeclareMathOperator{\Id}{Id}
\DeclareMathOperator{\SU}{SU}

\title{Sum-product for real Lie groups}
\date{}
\author{Weikun He \\ \small{Einstein Institute of Mathematics, The Hebrew University of Jerusalem} \and Nicolas de Saxcé \\ \small{CNRS, Université Paris 13, Villetaneuse}}

\begin{document}

\maketitle

\begin{abstract}
We prove a discretized sum-product theorem for representations of Lie groups whose Jordan-Hölder decomposition does not contain the trivial representation.
This expansion result is used to derive a product theorem in perfect Lie groups.
\end{abstract}
\section{Introduction}

Throughout this paper, $G$ will denote a connected real Lie group, endowed with a left-invariant Riemmanian metric.
For $x \in G$ and $\rho > 0$, we denote by $B_G(x,\rho)$ the ball of center $x$ and radius $\rho$ in $G$. For $A \subset G$ and $\rho > 0$, $A^{(\rho)}$ stands for the $\rho$-neighborhood of $A$ and $\NR{A}$ stands for the covering number of $A$ by $\rho$-balls, i.e.
\[\NR{A} = \min \bigl\{ N \in \N \mid \exists \, x_1,\dotsc, x_N \in G, A \subset \bigcup_{i=1}^N B_G(x_i,\rho)\bigr\}.\]
The same notation is used for other metric spaces.

\subsection{Sum-product theorem in representations of Lie groups}
In the first part of this paper, we study the sum-product phenomenon in representations of Lie groups.
We shall work with some linear representation of $G$ over some finite-dimensional real vector space $V$, endowed with some norm.
We shall also refer to
representations of $G$ as $G$-modules. For $A\subset G$, $X\subset V$ and $s\geq 1$, we denote by $\bracket{A,X}_s$ the set of elements in $V$ that can be obtained as combinations of sums, differences and products of at most $s$ elements from $A$ and $X$.

Note that the distance on $G$ induces a natural distance on each of its quotients.
Let $N \lhd G$ be a closed normal subgroup. We denote by $\pi_{G/N} \colon G \to G/N$ the canonical projection. Then there is a unique distance on $G/N$ satisfying $\forall x, y \in G,\ d(\pi_{G/N}(x), \pi_{G/N}(y)) = d(x^{-1}y,N)$. Throughout this paper, all quotients $G/N$ will be endowed with this distance.

Following~\cite{saxce_producttheorem}, we say that a subset $A \subset G$ is \emph{$\rho$-away from closed connected subgroups} for some parameter $\rho > 0$ if for any proper closed connected subgroup $H < G$, there exists $a \in A$ with $d(a,H) > \rho$.
Similarly, we say that a subset $X \subset V$ is \emph{$\rho$-away from submodules} if for any proper $G$-submodule $W < V$, there exists $x \in X$ with $d(x,W) > \rho$.

We say that a $G$-module $V$ is in $\cP(G)$ if the trivial representation does not appear as a simple quotient in the Jordan-Hölder decomposition of $G$ -- see Definition~\ref{pg}.
\begin{theorem}[Sum-product theorem in representations of class $\cP$]
\label{spclasspi}
Let $G$ be a connected real Lie group and $V\in\cP(G)$. There exists a neighborhood $U$ of the identity in $G$ such that,
for every $\eps_0,\kappa>0$,
there exist $s\geq 1$ and $\eps>0$ such that the following holds for any $\delta>0$ sufficiently small.
Assume $A \subset U$ and $X \subset B_V(0,1)$ satisfy:
\begin{enumerate}
\item For any proper closed connected normal subgroup $N \lhd G$,
\[\forall \rho\geq\delta,\ N(\pi_{G/N}(A),\rho) \geq \delta^\eps\rho^{-\kappa};\]
\item $A$ is $\delta^\eps$-away from closed connected subgroups;
\item $X$ is $\delta^\eps$-away from submodules.
\end{enumerate}
Then,
\[B_V(0,\delta^{\eps_0}) \subset \bracket{A,X}_s^{(\delta)}.\]
\end{theorem}

This is a bounded generation statement and can be used to recover previous sum-product results in rings.
For example, applying to $G = \R^*$ acting on $V =\R$, one recovers a version Bourgain's discretized sum-product theorem~\cite{Bourgain2003,Bourgain2010}.

In Section 2, Theorem~\ref{spclasspi} will be proved in a more precise form where the conditions (i) and (ii) can be slightly relaxed. See Theorem~\ref{spclassp}.

\subsection{Product theorem in perfect Lie groups}
In the second part of this paper, we use Theorem~\ref{spclasspi} to derive a product theorem in perfect Lie groups.
For subsets $A,B \subset G$ of a Lie group $G$, we denote by $AB$ their product set, i.e.
\[AB = \{ab \mid a \in A, b \in B\}.\]
For $k \geq 2$, we denote by $A^k$ the $k$-fold product set of $A$ with itself, $A \dotsm A$. To avoid confusion with Cartesian products between sets, we write $A^{\times k}$ for the Cartesian power $A \times \dotsm \times A$.  

Recall that a Lie group is perfect if its Lie algebra $\g$ is perfect, i.e. satisfies $[\g,\g] = \g$. 

\begin{theorem}[Product theorem in perfect Lie groups]
\label{producttheorem0}
Let $G$ be a connected perfect Lie group.
There exists a neighborhood $U$ of the identity in $G$ such that given $\kappa > 0$, there exists $\eps>0$ such that the following holds for $\delta>0$ sufficiently small.
Let $A$ be a subset of $U$ such that:
\begin{enumerate}
\item $\ND{A} \leq \delta^{-\dim G + \kappa}$;
\item for any proper closed connected normal subgroup $N \lhd G$,
\[\forall \rho\geq\delta,\ N(\pi_{G/N}(A),\rho) \geq \delta^\eps\rho^{-\kappa};\]
\item $A$ is $\delta^\eps$-away from closed connected subgroups in $G$.
\end{enumerate}
Then
\[\ND{AAA} \geq \delta^{-\eps}\ND{A}.\]
\end{theorem}

For $G = \SU(2)$, the above theorem was proved by Bourgain-Gamburd~\cite{BourgainGamburd_SU2}, and for a general simple Lie group by the second author~\cite{saxce_producttheorem}, borrowing many ideas from the work of Bourgain and Gamburd \cite{BourgainGamburd_SUd} on the spectral gap property in $\SU(d)$.

It is not difficult to see that the assumption of perfectness is optimal for such a product theorem to hold, in the following sense.

\begin{proposition}
\label{perfectisoptimal}
Let $G$ be a simply connected Lie group which is not perfect with Lie algebra $\g$. Write $d = \dim \g - \dim [\g,\g]$. For any neighborhood $U$ of the identity in $G$, for any $\kappa \in {(0, 1)}$ and for any $\delta > 0$ small enough, there exists $A \subset U$ such that
\begin{enumerate}
\item $\ND{A} \approx_U \delta^{-\dim G + d(1-\kappa)}$;
\item for any proper closed connected normal subgroup $N \lhd G$,
\[\forall \rho\geq\delta,\ N(\pi_{G/N}(A),\rho) \gg_U \rho^{-\kappa};\]
\item $A$ is $\frac{1}{O_U(1)}$-away from closed connected subgroups in $G$.
\end{enumerate}
but
\[\ND{AAA} \ll_U \ND{A}.\]
\end{proposition}

Note that in a closely related setting, Salehi-Golsefidy and Varj\'u~\cite{GolsefidyVarju} have already observed that perfectness is a sufficient and necessary condition for an expansion result to hold.
In fact, at different places, our arguments share some conceptual similarities with the recent work of Salehi Golsefidy \cite{sg_1,sg_2} on super-approximation.
Also, these examples of approximate subgroups can be seen as discretized analogues of measurable subgroups of intermediate dimension whose existence is known in abelian Lie groups~\cite{ErdosVolkmann} and solvable Lie groups~\cite{Saxce_nilpotent}.

We shall prove Theorem~\ref{producttheorem0} and Proposition~\ref{perfectisoptimal} in Section~\ref{sec:perfect}.

\subsection{Applications}

We conclude this introduction by mentioning several applications to Theorems~\ref{spclasspi} and \ref{producttheorem0} above.
The first is that the spectral gap property discovered by Bourgain and Gamburd~\cite{BourgainGamburd_SU2,BourgainGamburd_SUd} in the setting of $\SU(d)$, $d\geq 2$ generalizes to all compact semisimple Lie groups.

\begin{theorem}
Let $G$ be a compact connected semisimple Lie group, with Lie algebra $\g$, and write $L^2_0(G)$ for the space of zero-mean square-integrable functions on $G$.
Let $\mu$ be a probability measure on $G$ whose support generates a dense subgroup in $G$.
Assume moreover that in some basis for $\g$, for every $g\in\Supp\mu$, the matrix of $\Ad g$ has algebraic entries.
Then the convolution operator
\[ \begin{array}{cccc}
T_\mu: & L^2_0(G) & \to & L^2_0(G)\\
& f & \mapsto & f*g
\end{array} \]
satisfies $\|T_\mu\|<1$.
\end{theorem}

The local spectral gap property introduced by Boutonnet, Ioana and Salehi Golsefidy in \cite{bisg} for non-compact Lie groups can also be generalized to a general connected perfect Lie group, but in order to keep statements as elementary as possible, we do not make this precise here.

\bigskip

Originally, discretized expansion statements were introduced by Katz and Tao \cite{katztao} and used by Bourgain \cite{Bourgain2003} to study Hausdorff dimensions of sum-sets and product-sets in $\R$ and give a quantitative solution to the Erd\H{o}s-Volkmann conjecture: If $A$ is any Borel measurable subset of $\R$ with $\dim_HA\in(0,1)$, then $\dim_HA+AA>\dim_HA$.
The theorems proven here have similar consequences on the Haudorff dimension of sum and product sets in semi-simple algebras or perfect Lie groups.
In particular, it should be possible to generalize the results of the second author presented in \cite{saxce_simplesubgroups} to the setting of perfect Lie groups; we hope to address these matters in another paper.

\bigskip

Another nice application of Theorem~\ref{spclasspi} is the very recent work of Li \cite{li_fourierdecayrn, li_fourierdecayfurstenberg} on the regularity of the Furstenberg measure associated to a random walk on a semisimple Lie group.

\bigskip

Finally, it is possible that our results could be used to construct new family of expanders, in the spirit of the works of Bourgain-Yehudayoff~\cite{bourgain-yehudayoff} or Vigolo~\cite{vigolo_expanders}.



\section{Sum-product theorem in representations of $G$}

The goal of this section is to prove Theorem~\ref{spclasspi} from the introduction.
In fact, our proof will yield a slightly more precise version, see Theorem~\ref{spclassp}.

\subsection{Representations without trivial simple quotients}

We now define the class of representations to which our theorem will apply, and gather some elementary properties.
Then, we state the refined version of Theorem~\ref{spclasspi} which will be proved here, Theorem~\ref{spclassp}.

\begin{definition}
\label{pg}
Let $G$ be a connected Lie group.
We let $\cP(G)$ denote the class of finite-dimensional linear representations $V$ of $G$ such that
there exists a sequence $\{0\}=V_0<V_1<\dots<V_\ell=V$ of subrepresentations of $V$ such that, for each $i=0,\dotsc,\ell-1$, the quotient representation $V_{i+1}/V_i$ is \emph{non-trivial} and irreducible.
\end{definition}

Equivalently, $V$ is in $\cP(G)$ if the trivial representation does not appear as a simple quotient in a Jordan-Hölder decomposition of $V$.
This property, of course, does not depend on the choice of the Jordan-Hölder decomposition.
We now list some elementary properties of representations in $\cP(G)$.

\begin{proposition}
\label{pr:subquo}
Let $V$ be a representation of a connected Lie group $G$.
\begin{enumerate}
\item \label{it:subquo} If $W$ is a subrepresentation of $V$, then $V$ belongs to $\cP(G)$ if and only if both $W$ and $V/W$ belong to $\cP(G)$.
\item \label{it:asHrep} If $H$ is a closed subgroup of $G$ and $V \in \cP(H)$ as a representation of $H$, then $V \in \cP(G)$.
\item \label{it:asGmodHrep} Let $H$ be a normal subgroup of $G$.
If the representation $G \to \GL(V)$ factors through $G/H$, then $V \in \cP(G/H)$ as a representation of $G/H$ if and only if $V \in \cP(G)$ as a representation of $G$.
\end{enumerate}
\end{proposition}
\begin{proof}
Indeed, \ref{it:subquo} follows from the fact that the set of simple quotients of the Jordan-Hölder decomposition of $V$ is the union of those of $W$ and those of $V/W$.
For \ref{it:asHrep}, note that a Jordan-Hölder sequence of $G$-submodules in $V$ can be refined to a Jordan-Hölder sequence of $H$-submodules, and that
if there is a trivial quotient in the first sequence there must be also one in the refined sequence.
Finally, \ref{it:asGmodHrep} is clear, since Jordan-Hölder decompositions of $V$ into $G$-modules coïncide with Jordan-Hölder decompositions into $G/H$-modules.
\end{proof}

\begin{remark}
The class $\cP(G)$ is the smallest class of finite-dimensional $G$-modules that
contains all non-trivial irreducible representations of $G$ and is closed under extension (i.e., if $W$ and $V'$ are in $\cP(G)$, and $0\to W\to V\to V'\to 0$ is a short exact sequence of $G$-modules, then $V$ is in $\cP(G)$).
\end{remark}

\begin{example}
\begin{itemize}
\item If a representation $V$ contains the trivial representation, then it is not in $\cP(G)$. Similarly, if $V$ admits the trivial representation as a quotient, then it is not in $\cP(G)$.
\item Let $n$ be a positive integer. The representation of $G=\R_+^*$ on $\R^n$ given by $g\cdot v = gv$ (scalar multiplication) is in $\cP(G)$.
\item The adjoint representation of a semisimple Lie group $G$ is in $\cP(G)$.
\end{itemize}
\end{example}

Throughout this article, we shall consider representations of $G$ as normed vector spaces: By \emph{normed $G$-module}, we mean a $G$-module endowed with a norm which makes the underlying linear space a normed vector space.

Whenever $V$ is a normed vector space, and $W\leq V$ is a linear subspace, we shall alway consider on $W$ the norm induced by the norm on $V$, and on quotient space $V'=V/W$ the norm given by the formula
\[\forall v\in V,\ \norm{\pi(v)} = d(v,W),\]
where $\pi \colon V \to V'$ is the canonical projection.
Finally, we endow the space of linear endomorphisms $\End(V)$ with the associated operator norm.

Let $\rho \in(0,\frac{1}{2})$ be a parameter and $V$ a normed $G$-module, we say that a subset $X \subset V$ is \emph{$\rho$-away from submodules} if for every proper submodule $W < V$, there exists $x \in X$ such that $d(x,W) \geq \rho$.
Similarily, a subset $A\subset G$ is said to be \emph{$\rho$-away from closed connected subgroups} if for every proper closed connected subgroup $H$, there exists $a\in A$ such that $d(a,H)\geq\rho$.
Finally, a subset $A\subset G$ is said to be \emph{$\rho$-away from identity components of proper stabilizers} if for any subspace $W \subset V$ which is not a $G$-submodule, there exists $a \in A$ such that $d(a,(\Stab_GW)^\circ) \geq \rho$, where $(\Stab_GW)^\circ$ denotes the identity component of the stabilizer $\Stab_GW$ of $W$ in $G$.

We shall prove the following.

\begin{theorem}[Sum-product in representations of class $\cP$]
\label{spclassp}
Let $G$ be a connected real Lie group and $V$ a normed $G$-module. There exists a neighborhood $U$ of the identity in $G$ such that,
for every $\eps_0,\kappa>0$,
there exist $s\geq 1$ and $\eps>0$ such that the following holds for any $\delta>0$ sufficiently small.\\
Assume $A \subset U$ and $X \subset B_V(0,1)$ satisfy:
\begin{enumerate}
\item \label{it:piArich} There is a Jordan-Hölder sequence $0 = V_0 < \dotso < V_\ell = V$ such that for every $i = 0,\dotsc, \ell - 1$,
\[ \forall \rho\geq\delta, \quad N(p_{V_{i+1}/V_{i}}(A),\rho) \geq \delta^\eps\rho^{-\kappa},\]
where $p_{V_{i+1}/V_{i}} \colon G\to \GL(V_{i+1}/V_{i})$ denotes the representation of $G$ on $V_{i+1}/V_{i}$;
\item \label{it:Aaway} $A$ is $\delta^\eps$-away from identity components of proper stabilizers;
\item \label{it:Xaway} $X$ is $\delta^\eps$-away from submodules.
\end{enumerate}
Then,
\[B_V(0,\delta^{\eps_0}) \subset \bracket{A,X}_s + B_V(0,\delta).\]
\end{theorem}

Note that the assumption~\ref{it:piArich} implies that $V$ is of class $\cP(G)$.
The proof goes by induction on the length of $V$ (i.e. the length of any Jordan-Hölder decomposition of $V$).
We shall prove the base case, where $V$ is a non-trivial irreducible representation, in the next subsection.
The induction step will then be carried out in Subsection~\ref{ss:inductionstep}.

\subsection{Irreducible representations}

In the case $V$ is an irreducible representation of $G$, the above theorem is a variant of \cite[Theorem~3]{he_sumproduct}.
For clarity, we restate our theorem in this particular case.
Then, we shall explain how to derive it from the first author's sum-product theorem in simple algebras \cite[Theorem~2]{he_sumproduct}.

\begin{theorem}[Base case: irreducible representations]
\label{basecase}
Let $G$ be a connected real Lie group and $\pi_V:G\to\GL(V)$ a non-trivial irreducible representation.
There exists a neighborhood $U$ of the identity in $G$ such that, for every $\eps_0,\kappa>0$,
there exist $s\geq 1$ and $\eps>0$ such that the following holds for any $\delta>0$ sufficiently small.\\
Assume $A \subset U$ and $X \subset B_V(0,1)$ satisfy:
\begin{enumerate}
\item For every $\rho\geq\delta$, $N(\pi_V(A),\rho) \geq \delta^\eps\rho^{-\kappa}$;
\item $A$ is $\delta^\eps$-away from identity components of proper stabilizers;
\item There exists $v\in X$ such that $\|v\| \geq\delta^\eps$.
\end{enumerate}
Then,
\[B_V(0,\delta^{\eps_0}) \subset \bracket{A,X}_s + B_V(0,\delta).\]
\end{theorem}

The proof of this theorem is based on Proposition~\ref{pr:sumproduct} below, a sum-product statement in matrix representations, which is essentially contained in \cite{he_sumproduct}.
Below, $A$ denotes a subset of $\End(V)$, for some real vector space $V$, and $\bracket{A}_s$ denotes the set of elements in $\End(V)$ that can be obtained as combinations of sums and products of at most $s$ elements in $A$.
If $V$ is a real vector space, if $A$ is a subset of $\End V$, and if $\rho\in(0,\frac{1}{2})$ is a parameter, we say that $A$ \emph{acts $\rho$-irreducibly} on $V$ if for every non-trivial proper linear subspace $W<V$, there exists $v\in B_W(0,1)$ and $a\in A$ such that $d(a\cdot v,W)\geq\rho$.

\begin{proposition}[Sum-product in irreducible representations]
\label{pr:sumproduct}
Let $V$ be a finite-dimensional normed vector space. Given $\eps_0,\kappa>0$, there exist $s\geq 1$ and $\eps>0$ such that the following holds.
Let $A\subset B_{\End(V)}(0,\delta^{-\eps})$ be a subset of $\End V$ and $v \in V$ a vector. Assume that
\begin{enumerate}
\item \label{it:ar} For every $\rho\geq\delta$, $N(A,\rho) \geq \delta^\eps\rho^{-\kappa}$;
\item \label{it:airr} $A$ acts $\delta^\eps$-irreducibly on $V$;
\item \label{it:vd} $\delta^\eps \leq \norm{v} \leq \delta^{-\eps}$.
\end{enumerate}
Then
\[B_V(0,\delta^{\eps_0}) \subset \bracket{A}_s \cdot v + B_V(0,\delta).\]
\end{proposition}
\begin{proof}
Given $\eps_1>0$, it follows from \cite[Proposition 31]{he_sumproduct} that there exists $c>0$ such that, provided $\eps>0$ is small enough, there exists a $\delta^{-O(\eps)}$-bi-Lipschitz linear bijection $f \colon V \to K^n$, where $K$ is $\R$, $\C$ or the quaternions $\mathbb{H}$, $n$ is $\frac{\dim V}{\dim K}$ and $K^n$ is endowed with its usual $L^2$ norm, and a scale $\delta_1$ with $\delta \leq \delta_1 \leq \delta^c$ such that
\[fAf^{-1} \subset \Mat_n(K)+B(0,\delta_1)\]
and such that for every proper real subalgebra $F<\End(K^n)$,
\[ \exists a\in A: \ d(faf^{-1},F) \geq \delta_1^{\eps_1}.\]
Choosing $\eps_1$ small enough in terms of $\eps_0$ and $\kappa$, we may then apply \cite[Theorem 5]{he_sumproduct} to conclude that, provided $\eps>0$ is sufficiently small, for some integer $s$,
\[B_{\Mat_n(K)}(0,\delta_1^{\eps_0}) \subset f\bracket{A}_sf^{-1} + B_{\Mat_n(K)}(0,\delta_1).\]
Therefore, without loss of generality, we may assume that $V = K^n$ and
\begin{equation} \label{eq:MatnK}
B_{\Mat_n(K)}(0,\delta_1^{\eps_0}) \subset A + B_{\Mat_n(K)}(0,\delta_1).
\end{equation}
We can further assume that $\norm{v} = 1$. Then
\[B_V(0,\delta_1^{\eps_0}) \subset A\cdot v + B_V(0,\delta_1).\]
In other words, the conclusion of the proposition holds at scale $\delta_1$.
It remains to bring the scale back to $\delta$.
To do this, we note that from \eqref{eq:MatnK}, we have in particular
\[\delta_1^{\frac{1}{2}}\id \in A + B_{\Mat_n(K)}(0,\delta_1).\]
Hence, starting from (\ref{eq:MatnK}), we may multiply both sides by $\delta_1^{\frac{1}{2}}\id$ to obtain
\[ B_V(0,\delta_1)\subset B_V(0,\delta_1^{\eps_0+\frac{1}{2}}) \subset \bracket{A}_2 \cdot v + B_V(0,2\delta_1^{\frac{3}{2}}),\]
and iterating this procedure, we get a sequence of integers $s_2 = 1, s_3 = 2, s_4 , \dotsc $ such that for any $k \geq 2$,
\[B_V(0,s_k\delta_1^{\frac{k}{2}}) \subset \bracket{A}_{s_{k+1}} \cdot v + B_V(0,s_{k+1}\delta_1^{\frac{k+1}{2}}).\]
Choose $k>\frac{2}{c}$ so that $s_k\delta_1^{\frac{k}{2}} \leq \delta$. Combining all these inclusions, we find, for $s=s_2+\dotsb+s_k$,
\[B_V(0,\delta_1^{\eps_0}) \subset \bracket{A}_{s} \cdot v + B_V(0,\delta),\]
which proves the proposition.
\end{proof}

The above proposition readily implies Theorem~\ref{basecase}.

\begin{proof}[Proof of Theorem~\ref{basecase}]
It suffices to apply Proposition~\ref{pr:sumproduct} to the set $\pi_V(A)\subset\End(V)$.
By the assumption on $A$, conditions \ref{it:ar} and \ref{it:vd} of the proposition are satisfied for the set $\pi_V(A)$.
That condition \ref{it:airr} is also satisfied is a consequence of Lemma~\ref{lm:rhoOirred} below.
\end{proof}

\begin{lemma}\label{lm:rhoOirred}
Let $0 < \rho < \frac{1}{2}$ be a parameter. Let $\pi \colon G \to \GL(V)$ be a non-trivial irreducible representation.
There is a neighborhood $U$ of $1$ in $G$ such that if $A \subset U$ is $\rho$-away from identity components of proper stabilizers then $\pi(A)$ acts $\rho^{O_\pi(1)}$-irreducibly on $V$.
\end{lemma}

The proof of this lemma is an application of \L ojasiewicz's inequality, but first, it is convenient to reduce to the case where $A$ is finite.
This reduction is the subject of the next lemma.
Given a representation $\pi \colon G \to \GL(V)$ of $G$, a subset $A \subset G$ and a parameter $\rho\in(0,\frac{1}{2})$, we say that $A$ is \emph{$\rho$-away from proper stabilizers} if for any linear subspace $W$ of $V$ which is not a $G$-submodule, there exists an element $a$ in $A$ whose distance to the stabilizer $\Stab_GW$ is at least $\rho$.

\begin{lemma}\label{lm:Acanbefinite}
Let $0 < \rho < \frac{1}{2}$ be a parameter. Let $\pi \colon G \to \GL(V)$ be a representation. There is a neighborhood $U$ of $1$ in $G$ such that if $A \subset U$ is $\rho$-away from identity components of proper stabilizers then $A$ is $\rho^{O_\pi(1)}$-away from proper stabilizers.
In fact, $A$ contains a subset of cardinality at most $\dim G$ which is $\rho^{O_\pi(1)}$-away from proper stabilizers.
\end{lemma}
\begin{proof}
The representation $\pi$ differentiates to a representation of the Lie algebra $\g$ of $G$, which we denote by $T_{\!1}\pi \colon \g \to \End(V)$.
The stabilizer of $W$ in $\g$
\[\Stab_\g W = \{x \in \g \mid T_{\!1}\pi(x)W \subset W\}\]
is the Lie algebra of $\Stab_GW$.
In particular, its image under the exponential map is contained in $(\Stab_GW)^\circ$, the identity component of $\Stab_GW$.
We may assume that $\exp$ induces a diffeomorphism from $U$ to its image, and denote the inverse map by $\log$.
Say that $\log A$ is \emph{$\rho$-away from proper stabilizers} in $\g$ if for any linear subspace $W < V$ which is not a $G$-submodule, there exists $a \in A$ such that $d(\log a,\Stab_\g W) \geq \rho$.

We claim that there is a neighborhood $U$ of $1$ in $G$ such that if $A \subset U$ is $\rho$-away from identity components of proper stabilizers then $\log A$ is $\frac{\rho}{C}$-away from proper stabilizers in $\g$ and conversely if $\log A$ is $\rho$-away from proper stabilizers then $A$ is $\frac{\rho}{C}$-away from proper stabilizers.

Let us prove this claim.
Let $x \in \g$. From the identity $\pi(e^x) = e^{T_{\!1}\pi(x)}$, we can express $T_{\!1}\pi(x)$ as an absolutely convergent series
\[T_{\!1}\pi(x) = - \sum_{n\geq 1} \frac{1}{n}\bigl(\id_V - \pi(e^x)\bigr)^n\]
whenever $\norm{\pi(e^x) - \id_V} < 1$.
Therefore, if $e^x\in\Stab_G W$ is such that the above series converges, then $x\in\Stab_\g W$.
It follows that there is $r > 0$ depending only on $\pi$ such that
\[ B_G(1,r)\cap \Stab_G W \subset \exp(\Stab_\g W).\]
Let $U = B_G(1,\frac{r}{2})$. Then for any $g \in U$ and any proper linear subspace $W$,
\[\frac{1}{C}d(g,(\Stab_G W)^\circ) \leq d(\log g,\Stab_\g W) \leq Cd(g,\Stab_G W)\]
where $C> 0$ is some constant depending only on the the representation. This proves our claim, and the first part of the lemma.

For the second part, one can reproduce the argument in~\cite[Lemma~2.5]{saxce_producttheorem} to show that if $\log A$ is $\rho$-away from proper stabilizers then $\log A$ contains a subset of cardinality at most $\dim \g$ which is $\rho^{O_{\dim(\g)}(1)}$-away from stabilizers.
\end{proof}

\begin{remark}
Note that in the above lemma, the neighborhood $U$ depends on the representation $\pi$, and not only on $G$.
This is readily seen by considering $G=\R$, $V=\C\simeq\R^2$, and $\pi(x)v=e^{inx}v$, $n \in \N$.
\end{remark}

\begin{proof}[Proof of Lemma~\ref{lm:rhoOirred}]
Let $U$ be the neighborhood given by Lemma~\ref{lm:Acanbefinite}. On account of this lemma we may assume that $A$ is finite of cardinality $n \leq \dim G$ and $\rho$-away from proper stabilizers. Let $0 < k < \dim(V)$ and consider the analytic map $f \colon G^n \times \Grass(k,V) \to \R$ defined by
\[f(g_1,\dotsc,g_n;W) = \sum_{i=1}^n\int_{\raisebox{-0.3ex}{$\scriptstyle B_W(0,1)$}} \hspace{-1em} d(g_i \cdot w,W)^2 \dd w.\]
The zero set of $f$ is exactly
\[Z = \{ (\mathbf{g},W) \in G^n \times \Grass(k,V) \ |\  \forall i,\, g_i \in \Stab_G W\}.\]
By \L{}ojasiewicz's inequality \cite[Théorème 2, page 62]{lojasiewicz} applied on $\bar{U}^n \times \Grass(k,V)$, there is a constant $C>0$ such that for any $(\mathbf{g},W) \in U^n \times \Grass(k,V)$,
\[f(\mathbf{g},W) \geq \frac{1}{C}d\bigl( (\mathbf{g},W),Z\bigr)^C.\]
Assuming that $A$ does not act $\frac{1}{C}\rho^C$-irreducibly on $V$, we can find $W \in \Grass(k,V)$ such that for all $a \in A$ and all $w \in B_W(0,1)$, $\pi(a)w \in W + B_V(0,\frac{1}{C}\rho^{C})$.
So $f(a_1,\dots,a_n,W)\leq\frac{1}{C}\rho^C$, and by the inequality above there exists $W' \in \Grass(k,V)$ such that for all $a \in A$, $d(a, \Stab_G W') \leq \rho$, so that $A$ is not $\rho$-away from proper stabilizers.
\end{proof}

\subsection{Induction step}\label{ss:inductionstep}

The core of the induction step in the proof of Theorem~\ref{spclassp} is the following lemma.
It is a quantitative discretized version of the following elementary fact: let $V$ be a $G$-module, and $V_1,X$ two submodules of $V$; if $\pi:V\to V/V_1$ maps $X$ onto $V/V_1$ and if $X\cap V_1=\{0\}$, then $V=X\oplus V_1$.
Once more, the proof relies on \L ojasiewicz's inequality.

\begin{lemma}\label{inductionstep}
Let $G$ be a connected Lie group and $V$ a normed $G$-module.
There exist a neighborhood $U$ of the identity in $G$ and a constant $C\geq 1$ such that for any parameters $0 < \eta < \tau < 1$, the following holds when $\delta$ is sufficiently small.
Let $V_1$ be a proper submodule of $V$ and $\pi \colon V\to V/V_1$ the canonical projection.
Let $A\subset U$ and $X\subset B_V(0,1)$ and assume that
\begin{enumerate}
\item \label{it:AXWinB} $\bracket{A,X}_3 \cap V_1^{(\delta)} \subset B_V(0,\delta^{C\tau})$,
\item \label{it:piXball} $\pi(X) = B_{V/V_1}(0,\delta^\eta)$,
\item $A$ is $\delta^\tau$-away from identity components of proper stabilizers.
\end{enumerate}
Then there exists a submodule $W < V$ such that:
\begin{enumerate}[label=(\alph*)]
\item \label{eq:dangV1W} The restriction $\pi_{\mid W} \colon W \to V/V_1$ is $3\delta^{-\eta}$-bi-Lipschitz;
\item \label{eq:XsubmoduleV1} $B_{W}(0,\delta^\eta) \subset X^{(\delta^\tau)}$  and $X \subset W^{(\delta^\tau)}$.
\end{enumerate}
\end{lemma}

\begin{proof}
For convenience, we write $V'=V/V_1$. On account of Lemma~\ref{lm:Acanbefinite}, which gives us the neighborhood $U$, we may assume that $A$ is finite of cardinality $n \leq \dim(G)$ and is $\delta^{C_1\tau}$-away from proper stabilizers, where $C_1 \geq 2$ is a constant depending only on $V$. Shrinking again the neighborhood $U$ if necessary, we can ensure that the action on $V$ of any element in $A$ is $2$-bi-Lipschitz.

Assumption~\ref{it:piXball} allows us to pick a section $\sigma : B_{V'}(0,\delta^\eta) \to X$ of the projection $\pi$, i.e. for any $y\in B_{V'}(0,\delta^\eta)$,
\[\pi\circ\sigma(y)=y.\]
The choice of such $\sigma$ is arbitrary. In fact, thanks to assumption~\ref{it:AXWinB}, different choices only differ by at most $\delta^{C\tau}$. Indeed, for any $x \in X$, we have $x - \sigma(\pi(x)) \in (X - X) \cap V_1$ and therefore, by assumption~\ref{it:AXWinB},
\begin{equation}\label{eq:sigmaUni}
\norm{x - \sigma(\pi(x))} \leq \delta^{C\tau}.
\end{equation}
Again from assumption~\ref{it:AXWinB}, it follows that $\sigma$ is almost a morphism of $G$-modules, in the sense that for all $y,z \in B_{V'}(0,\delta^\eta)$ and all $a \in A$,
\begin{align}
\label{eq:sigma00} \norm{\sigma(y)} &\leq \delta^{C\tau} \quad\text{ if } y \in B_{V'}(0,\delta); \\
\label{eq:sigmadd} \norm{\sigma(y) + \sigma(z) - \sigma(y + z)} &\leq \delta^{C\tau} \quad\text{ if } y+z \in B_{V'}(0,\delta^\eta);\\
\label{eq:sigmaeq} \norm{a \cdot \sigma(y) - \sigma(a \cdot y)}  &\leq \delta^{C\tau} \quad\text{ if } a \cdot y \in B_{V'}(0,\delta^\eta).
\end{align}
Indeed, we have, respectively, $\sigma(y) \in X \cap V_1^{(\delta)}$, $\sigma(y) + \sigma(z) - \sigma(y + z) \in 3X \cap V_1$ and $a \cdot \sigma(y) - \sigma(a \cdot y) \in (A \cdot X - X)\cap V_1$.

In particular, \eqref{eq:sigma00} and \eqref{eq:sigmadd} says that $\sigma$ is almost additive; by Lemma~\ref{almostadditive} below, $\sigma$ is close to a genuine linear map.
More precisely, there exists a linear section $\phi : V' \to V$ of $\pi$ (i.e. $\pi \circ \phi = \Id_{V'}$) such that for all $y\in B_{V'}(0,\delta^\eta)$,
\begin{equation}\label{eq:phisigma}
 \|\phi(y)-\sigma(y)\| \leq\delta^{(C-1)\tau},
\end{equation}
provided $\delta$ is small enough.
From the linearity of $\phi$, the fact that $X \subset B_V(0,1)$, and \eqref{eq:sigmaUni}, \eqref{eq:sigmaeq} and \eqref{eq:phisigma}, we obtain that for all $y \in V'$, all $a \in A$ and all $x \in X$,
\begin{equation*}
\norm{\phi(y)} \leq  2\delta^{-\eta} \norm{y};
\end{equation*}
\begin{equation*}
\norm{a \cdot \phi(y) - \phi(a \cdot y)} \leq \delta^{(C - 3)\tau}\norm{y};
\end{equation*}
\begin{equation*}
\norm{x - \phi(\pi(x))} \leq \delta^{(C-2)\tau}.
\end{equation*}

Let $W_0$ be the image subspace of $\phi$. From the above, it follows that:
\begin{equation}\label{eq:dangV1W0}
\text{the restriction } \pi_{\mid W_0} \colon W_0 \to V' \text{ is $2\delta^{-\eta}$-bi-Lipschitz};
\end{equation}
\begin{equation}\label{xinvo}
X \subset W_0 + B_V(0,\delta^{(C-2)\tau});
\end{equation}
\begin{equation}\label{voinx}
B_{W_0}(0,\delta^\eta/2)
\subset \varphi(B_{V'}(0,\delta^\eta))
\subset X + B_V(0,\delta^{(C-1)\tau});
\end{equation}
\begin{equation}\label{eq:avV0close}
\forall a\in A,\,\forall w \in B_{W_0}(0,1),\quad d(a\cdot w,W_0) \leq \delta^{(C - 3)\tau}.
\end{equation}

The inequality~\eqref{eq:avV0close} says that $W_0$ is almost invariant under the action of $A$. We now use \L{}ojasiewicz's inequality to show that it is close to a $G$-submodule. Let $a_1,\dotsc,a_n$ be the elements of $A$ and write $\ba = (a_1,\dotsc,a_n)$. Consider the real-analytic function on $G^{\times n}\times\Grass(\dim(V'),V)$ defined by
\[ f(g_1,\dotsc,g_n;W) = \sum_{i=1}^n\int_{\raisebox{-0.3ex}{$\scriptstyle B_{W}(0,1)$}} \hspace{-1em} d(g_i \cdot w,W)^2 \dd w.\]
From~\eqref{eq:avV0close} follows $f(\ba,W_0) \leq \delta^{(2C-7)\tau}$, provided $\delta$ is small enough.
By \L{}ojasiewicz's inequality \cite[Théorème 2, page 62]{lojasiewicz} applied to the compact set $\bar{U}^{\times d}\times\Grass(\dim(V'),V)$, there exists a constant $C_2$
depending only on the representation $V$
such that for all $\mathbf{g}=(g_1,\dotsc,g_n) \in U^{\times n}$ and $W \in \Grass(\dim(V'),V)$,
\begin{equation*}
 f(\mathbf{g},W) \geq \frac{1}{C_2}d((\mathbf{g},W),Z)^{C_2},
\end{equation*}
where $Z$ is the zero set of $f$.
Therefore, there exists $\bb=(b_1,\dotsc,b_n) \in G^{\times n}$ and $W \in \Grass(\dim(V'),V)$ such that $f(\bb,W)=0$ and
\[d((\ba,W_0),(\bb,W))\leq \delta^{C_1\tau},\] provided $2C-7 \geq (C_1 +1)C_2$.
The equality $f(\bb,W)=0$ exactly means that each $b_i$ belongs to the stabilizer $\Stab_GW$, and hence
\[A \subset (\Stab_GW)^{(\delta^{C_1\tau})}\]
But $A$ is $\delta^{C_1\tau}$-away from proper stabilizers, hence $W$ must be a $G$-submodule.
Finally, conclusions \ref{eq:dangV1W} and \ref{eq:XsubmoduleV1} follow from \eqref{eq:dangV1W0}, \eqref{xinvo}, \eqref{voinx} and the fact that $W$ is $\delta^{C_1\tau}$-close to $W_0$.
\end{proof}

In the above proof, we made use of the following elementary lemma, a discretized version of the fact that any continuous additive map between two vector spaces is automatically linear.

\begin{lemma}[Almost additive maps]
\label{almostadditive}
Let $0 < \delta<\rho_1<\rho_2 \leq 1$ be parameters.
Let $V$ and $V'$ be finite-dimensional normed vector spaces.
If $\sigma \colon B_{V'}(0,\rho_2) \to V$ satisfies
\begin{enumerate}
\item \label{it:almostCont} $\sigma(B_{V'}(0,\delta)) \subset B_V(0,\rho_1)$ and
\item \label{it:almostAdd} for all $x, y\in B_{V'}(0,\rho_2)$, if $x + y \in B_{V'}(0,\rho_2)$ then
\[\sigma(x) + \sigma(y) - \sigma(x+y) \in B_V(0,\rho_1).\]
\end{enumerate}
Then there is a linear map $\phi \colon V' \to V$ such that for all $x \in B_{V'}(0,\rho_2)$,
\[\|\sigma(x) - \phi(x)\| \ll_{V'} (-\log \delta + 1) \rho_1.\]
Moreover, if there are linear maps $\pi \colon V \to V''$ and $\psi \colon V' \to V''$ such that $\pi \circ \sigma = \psi$ on $B_{V'}(0,\rho_2)$, then we may also ensure that $\pi \circ \phi = \psi$ on $V'$.
\end{lemma}
\begin{proof}
We first consider the special case where $\rho_2=1$ and $V' = \R$. In this case define $\phi \colon \R \to V$ to be the unique linear map such that $\phi(1) = \sigma(1)$. From assumption~\ref{it:almostAdd}, it follows that
\[\forall x \in {[0,\frac{1}{2}]},\quad \norm{2\sigma(x) - \sigma(2x)} \leq \rho_1.\]
Using this and a simple induction, we prove that
\begin{equation}\label{eq:dyadicS-Ph}
\forall n \in \N,\quad \norm{\sigma(2^{-n}) - \phi(2^{-n})} \leq\rho_1.
\end{equation}
Let $N$ be the integer such that $2^{-N} \leq \delta < 2^{-N + 1}$. It follows from \eqref{eq:dyadicS-Ph} and assumption \ref{it:almostCont} that
\begin{equation}\label{eq:normPh}
\norm{\phi(2^{-N})} \leq 2\rho_1
\end{equation}
For any $x \in {[0,1]}$, let $(x_1,\dotsc,x_N) \in \{0,1\}^N$ be the $N$ first digits in its binary expansion, i.e. for some $r \in {[0,\delta]}$,
$x = \sum_{n= 1}^N x_n 2^{-n} + r$.
Then by assumption~\ref{it:almostAdd}, \eqref{eq:dyadicS-Ph} and \eqref{eq:normPh},
\begin{align*}
\norm{\sigma(x) - \phi(x)} &\leq \sum_{n=1}^N x_n\norm{\sigma(2^{-n}) - \phi(2^{-n})} + \norm{\sigma(r)} +  2^N r\norm{\phi(2^{-N})} + N\rho_1\\
&\leq (2N + 5)\rho_1.
\end{align*}
Consequently,
\begin{align*}
\norm{\sigma(-x) - \phi(-x)} &\leq \norm{\phi(x) -\sigma(x)} + \norm{\sigma(-x) + \sigma(x) - \sigma(0)} + \norm{\sigma(0)}\\
&\leq (2N + 7)\rho_1.
\end{align*}
This proves the lemma in the case $V'=\R$ and $\rho_2=1$.
For general normed vector space $V'$, in the case $\rho_2=1$, pick a basis $(u_1,\dotsc,u_d)$ consisting of vectors of unit length then apply the special case to each partial function $\sigma_i:t \mapsto \sigma(tu_i)$, $i = 1,\dotsc,d$.
This yields linear maps $\phi_1,\dots,\phi_d:\R\to V$, and we define $\phi:V'\to V$ by $\phi(t_1u_1+\dots+t_du_d)=\phi_1(t_1)+\dots+\phi_d(t_d)$.
Then by \ref{it:almostAdd}, we have the desired inequality for any vector in $B_{V'}(0,1)\cap({[-1,1]}u_1 + \dotsb + {[-1,1]}u_d)$. This domain contains a ball $B_{V'}(0,\frac{1}{k})$ where $k \in \N$ depends only on $V'$ and the choice of the basis. We conclude by using $k$ times the almost additivity \ref{it:almostAdd}.

The general case $\rho_2\leq 1$ follows from the case $\rho_2=1$, by considering the map $\sigma':V'\to V$ defined by $\sigma'(x)=\sigma(\rho_2x)$.

The moreover part is clear from the proof.
\end{proof}

%

We are now ready to prove Theorem~\ref{spclassp}.
The main idea is to use induction on the length of the module.
Note that among the assumptions of Theorem~\ref{spclassp}, \ref{it:Aaway} is preserved by passing to any submodule or any quotient of $V$ and \ref{it:Xaway} is preserved by passing to any quotient. Finally, \ref{it:piArich} passes to the quotient $V/V_1$ of $V$ by the first submodule $V_1$ in the Jordan-Hölder decomposition.
Thus by the induction hypothesis, it is easy to produce a large ball in $V/V_1$.
Then it can be proved (this is done in the third step of the proof below) that we can produce a large vector in $V_1$ and hence a large ball in $V_1$ by the base case. Then a technical difficulty arises : a large ball in $V/V_1$ and a large ball in $V_1$ does not add up to a ball in $V$. To deal with this difficulty we need to produce the large ball in $V/V_1$ using only vectors of controlled length (this is done in the first step in the proof below). Another technical difficulty is in the third step where we want to produce vector in $V_1$ of length $\geq \delta^{\eps_2}$ for any given $\eps_2 > 0$.
The idea is that, otherwise we could apply Lemma~\ref{inductionstep} to conclude that $X$ is trapped in a submodule, which would contradict assumption (\ref{it:Xaway}).

\begin{proof}[Proof of Theorem~\ref{spclassp}]
The proof goes by induction on the length $\ell$ of the module $V$.
The base case $\ell=1$, where $V$ is a non-trivial irreducible representation, corresponds to Theorem~\ref{basecase}, and is proved above.
Assume that the result holds for all representations of length less than $\ell\geq 2$, let $V\in\cP(G)$ be a representation of length $\ell$, and suppose $A\subset G$ and $X\subset V$ satisfy conditions \ref{it:piArich}-\ref{it:Aaway}-\ref{it:Xaway} of the theorem, for some small $\eps>0$.
Let $0 = V_0 < \dotso < V_\ell = V$ be the Jordan-Hölder sequence given by assumption~\ref{it:piArich}. Write $V'= V/V_1$ and denote by $\pi_{V'} \colon V \to V'$ the projection. Then the module $V'$ has length $\ell -1$ and as noted above, the conditions in Theorem~\ref{spclassp} are satisfied for $A$ acting on $\pi_{V'}(X) \subset V'$.


\noindent\underline{First step:} We first prove that there exists $\eps_1 > 0$ and $s_1 \geq 1$ depending on $V$, $\eps_0$ and $\kappa$ such that
\[B_{V'}(0,\delta^{\eps_0}) \subset \pi_{V'}(\bracket{A,X}_{s_1}\cap B(0,\delta^{\eps_1}))+B_{V'}(0,\delta).\]

Let $\eps_1\in(0,\eps_0)$ be a small parameter, whose precise value will be specified at the end of this step.
By applying the induction hypothesis to $V'$, whose length is at most $\ell-1$, and replacing $X$ by $\bracket{A,X}_s$, we may assume that $B_{V'}(0,\delta^{\eps_1}) \subset \pi_{V'}(X)^{(\delta)}$.
Cover $X$ with $\delta^{-O(\eps_1)}$ balls of radius $\delta^{2\eps_1}$, pick a ball $B$ such that $\ND{\pi_{V'}(B \cap X)}$ is maximal, and translate it back to the origin to get
\begin{equation*}
 \ND{\pi_{V'}(X')} \geq \delta^{-\dim(V') + O(\eps_1)},
\end{equation*}
with $X' = (X - X)\cap B_V(0,\delta^{2\eps_1})$.
This lower bound ensures that $\pi_{V'}(X')$ is $\delta^{O(\eps_1)}$-away from proper linear subspaces in $V'$.
The induction hypothesis, applied to the subset $\pi_{V'}(X') \subset V'$, with acting set $A$, yields the desired inclusion provided that $\eps_1$ is small enough.

\noindent\underline{Second step:} Assuming $X^{(\delta)} \cap V_1$ contains a large vector.

Let $s_2, \eps_2>0$ be the quantities given by Theorem~\ref{basecase} applied to the representation $V_1$, with constants $\kappa,\eps_1$.
We may choose $s_2$ and $\eps_2$ uniformly over all choices for $V_1$; indeed, up to a $(\dim V)$-bi-Lipschitz isomorphism of $G$-modules, there are only finitely many choices for $V_1$.
And assume that there exists $v\in X^{(\delta)} \cap V_1$ with
$\|v\|\geq\delta^{\eps_2}$.
Then, using the base case for the action of $G$ on the irreducible module $V_1$, we find that
\begin{equation}\label{inw}
B_{V_1}(0,\delta^{\eps_1}) \subset \bracket{A,X^{(\delta)}}_{s_2} + B_V(0,\delta).
\end{equation}
Now let $z\in B_V(0,\delta^{\eps_0})$.
By the first step, we may find $y\in\bracket{A,X}_{s_1}\cap B_V(0,\delta^{\eps_1})$ and $t\in V_1$ such that $z=y+t+O(\delta)$.
Necessarily, $\|t\|<2\delta^{\eps_1}$, so that by \eqref{inw}, $t\in\bracket{A,X}_{2s_2}+O(\delta)$.
All in all, setting $s=s_1 + O_{s_1}(s_2)$, we find
\[ B_V(0,\delta^{\eps_0}) \subset \bracket{A,X}_s+B_V(0,O_{s_1,s_2}(\delta)).\]
This finishes the proof of the theorem in this case.

\noindent\underline{Third step:}
Finally, we prove that there exists $s_3 \geq 1$ depending on $V$, $\eps_0$ and $\kappa$ such that $\bracket{A,X}_{s_3}^{(\delta)} \cap V_1$ always contains a vector of length at least $\delta^{\eps_2}$, which allows to conclude, using the second step.

Let $C$ be the constant given by Lemma~\ref{inductionstep}. Let $0 < \eps_3 < \frac{\eps_2}{C}$ be a parameter whose value will be chosen later according to $\eps_2$. Let $0 < \eps_4 < \eps_3$ be a parameter whose value will be chosen later according to $\eps_3$. Using the induction hypothesis for the representation $V'$ with $\eps_4$ and $\kappa$, and replacing $\bracket{A,X}_s^{(\delta)}$ by $X$, we may assume without loss of generality that
\begin{equation}\label{eq:Beps4piX}
B_{V'}(0,\delta^{\eps_4}) \subset \pi_{V'}(X).
\end{equation}
Either $\bracket{A,X}_3 \cap V_1^{(\delta)}$ contains a vector of length $\geq \delta^{\eps_2}$, in which case we are done or $\bracket{A,X}_3 \cap V_1^{(\delta)} \subset B_V(0, \delta^{\eps_2})$. In the latter case, Lemma~\ref{inductionstep} applied with $\tau = \frac{\eps_2}{C}$ and $\eta = \eps_4$ gives a submodule $W < V$ such that the restriction of $\pi_{V'}$ to $W$ is $3\delta^{-\eps_4}$-bi-Lipschitz and
\begin{equation}\label{eq:BWpiX}
B_{W}(0,\frac{1}{2}\delta^{\eps_4}) \subset X^{(\delta_1)}
\end{equation}
where $\delta_1=\delta^{\frac{\eps_2}{C}}$. Now we apply the base case, Theorem~\ref{basecase}, to the non-trivial irreducible representation $V/W$ with $\eps_3$ and $\kappa$.
Observe that $\pi_{V'\mid W}$ being $3\delta^{-\eps_4}$-bi-Lipschitz implies that $\pi_{V/W \mid V_1} \colon V_1 \to V/W$ is $4\delta^{-\eps_4}$-bi-Lipschitz.
Hence, for the projections $p_{V/W}:G\to\End(V/W)$ and $p_{V_1}:G\to\End(V_1)$, we have
\[\forall \rho \geq \delta, \quad \NR{p_{V/W}(A)} \geq \delta^{O(\eps_4)} \NR{p_{V_1}(A)} \geq \delta^{O(\eps_4) + \eps}\rho^\kappa.\]
Therefore, provided $\eps_4$ and $\eps$ are small enough in terms of $V_1$, $\eps_3$ and $\kappa$, Theorem~\ref{basecase} yields some constant $s \geq 1$ depending only on $V/W$, $\kappa$ and $\eps_3$ such that
\begin{equation*}
B_{V/W}(0,\delta^{\eps_3}) \subset \pi_{V/W}(\bracket{A,X}_{s}) + B_{V/W}(0,\delta).
\end{equation*}
Together with inclusion~\eqref{eq:BWpiX}, this implies that
\[ N\bigl(\bracket{A,X}_{s} + X,\delta_1\bigr) \gg (\delta_1^{-1}\delta^{\eps_3})^{\dim V/W} (\delta_1^{-1}\delta^{\eps_4})^{\dim W} \geq \delta_1^{-\dim V}\delta^{O(\eps_3)}.\]

Cutting $\bracket{A,X}_{s+1}$ into cylinders of axis $V_1$ and diameter $\delta^{\eps_3}$ and picking the part with largest size, we see that
\[N(\bracket{A,X}_{2s + 2}\cap V_1^{(\delta^{\eps_3})},\delta_1) \geq  \delta_1^{-\dim V}\delta^{O(\eps_3)},\]
which ensures that $X' := \bracket{A,X}_{2s + 2}\cap V_1^{(\delta^{\eps_3})}$ is $\delta^{O(\eps_3)}$-away from proper linear subspaces and a fortiori from submodules. Remembering \eqref{eq:Beps4piX}, we know that
\[\pi_{V'}(X') = B_{V'}(0,\delta^{\eps_3}).\]

At this stage apply Lemma~\ref{inductionstep} to the set $X'$ with $\tau = \frac{\eps_2}{C}$ and $\eta = \eps_3$. If $\eps_3$ is chosen sufficiently small compared to $\eps_2$, conclusion~\ref{eq:XsubmoduleV1} fails while all assumptions except~\ref{it:AXWinB} are satisfied.
So there must be $v\in\bracket{A,X^{\prime}}_3 \cap V_1^{(\delta)}$ with $\|v\| > \delta^{\eps_2}$.
This concludes the proof of the theorem.
\end{proof}

\section{A product theorem for perfect Lie groups}
\label{sec:perfect}

The goal of this section is to use Theorem~\ref{spclassp} to prove Theorem~\ref{producttheorem0}. More precisely, we prove the following essentially equivalent version of Theorem~\ref{producttheorem0}, which is a bounded generation statement.

\begin{theorem}[Product theorem in perfect Lie groups]
\label{producttheorem}
Let $G$ be a connected perfect Lie group.
There exists a neighborhood $U$ of the identity in $G$ such that given $\kappa > 0$ and $\eps_0>0$, there exist $\eps>0$ and $s \geq 1$ such that the following holds for $\delta>0$ sufficiently small.
Let $A$ be a subset of $U$ such that:
\begin{enumerate}
\item \label{it:richinSfactors} For any projection $\pi_i:G\to G/H_i$ to a simple factor,
\[\forall \rho \geq \delta, \quad \NR{\pi_i(A)} \geq \delta^\eps\rho^{-\kappa};\]
\item \label{it:AawayH} $A$ is $\delta^\eps$-away from closed connected subgroups in $G$.
\end{enumerate}
Then
\[ B_G(1,\delta^{\eps_0}) \subset (A\cup \{1\} \cup A^{-1})^s B_G(1,\delta).\]
\end{theorem}

Theorem~\ref{producttheorem0} follows immediately from Theorem~\ref{producttheorem} in combination with Ruzsa-type inequality \cite[Theorem 6.8]{taoestimates}.

The proof of Theorem~\ref{producttheorem} goes as follows.
We shall first prove the special case where the radical of our perfect Lie group $G$ is abelian.
In this case, the adjoint representation of $G$ belongs to $\cP(G)$, as we shall see in Lemma~\ref{abelianextension} below.
So Theorem~\ref{spclassp} applies and shows that we can produce in the Lie algebra $\g$ of $G$ a large ball using addition and the adjoint action of $G$: $B_\g(0,\delta^{\eps_0}) \subset \bracket{A,\log A}_s^{(\delta)}$.
Then we want to exponentiate this inclusion to the level of the group $G$.
For that, we use the Campbell-Hausdorff formula, which allows us to approximate sums in $\g$ by products in $G$ with any desired precision; this is the content of Lemma~\ref{lm:BCHnormed}.
Finally, to deduce the general case from the special case, we shall use an induction on the nilpotency class of the radical of $G$.


\subsection{Perfect Lie algebras and Lie groups}
\label{subsec:plag}
We begin by recording some elementary facts about perfect Lie groups and Lie algebras.

Let $G$ be a connected Lie group with Lie algebra $\g$.
Using Levi's decomposition theorem \cite[Corollary~1, p. 49]{serrelalg}, we may write $\g$ as a semi-direct product $\g=\s\ltimes\rf$ of a semi-simple Lie algebra $\s$ and a solvable radical $\rf$.
Writing $\s=\s_1\oplus \dotsb \oplus \s_k$ as a sum of simple ideals, one sees that for each $i$ in $\{1,\dots,k\}$, $\h_i=(\oplus_{j\neq i}\s_j)\ltimes\rf$ is an ideal in $\g$.
The Lie algebra $\h_i$ is the Lie algebra of a closed normal subgroup $H_i\lhd G$.
The projection maps $\pi_i:G\to G/H_i$ are the \emph{projections of $G$ to its simple factors}.
Note that any left-invariant Riemannian metric $d$ on $G$ induces a left-invariant metric on $G/H_i$.
Indeed, if $N\lhd G$ is any closed normal subgroup, one defines a distance on the quotient $G/N$ by
\[\forall x,y\in G,\quad d(\bar{x},\bar{y})= \inf_{n,n' \in N} d(xn,yn') = d(y,xN) = d(x^{-1}y, N).\]

For later use, we now list three elementary and standard lemmas about perfect Lie algebras.

\begin{lemma}\label{lm:radnil}
If $\g$ is a perfect Lie algebra, then its solvable radical $\rf$ is nilpotent.
In particular, $\g$ can be written as a semi-direct product $\g=\s\ltimes\rf$ of a semi-simple Lie algebra $\s$ with a \emph{nilpotent} ideal $\rf$.
\end{lemma}
\begin{proof}
See for instance \cite[Lemma~2.4]{benoistsaxce_convolution}.
\end{proof}

\begin{lemma}\label{lm:algFrattini}
Let $\g$ be a perfect Lie algebra, with Levi decomposition $\g=\s\ltimes\rf$.
The image of a proper ideal of $\g$ under the map $\g \to \g / \rf$ is a proper ideal.
In particular, the image of a maximal proper ideal is a maximal proper ideal.
\end{lemma}
\begin{proof}
Let $\nf$ be an ideal in $\g$ such that $\nf + \rf = \g$.
We want to show that $\nf=\g$.
Denote by $\Derive^i \rf$, $i \geq 0$ the derived series of $\rf$, i.e. $\Derive^0 \rf = \rf$ and $\Derive^{i + 1}\rf = [\Derive^i \rf; \Derive^i \rf]$, $\forall i \geq 0$.
We show by induction that $\forall i \geq 0$,
\begin{equation}\label{eq:g=n+Dr}
\g = \nf + \Derive^i\rf.
\end{equation}
Indeed, \eqref{eq:g=n+Dr} is true for $i = 0$.
Suppose that it is true for some $i \geq 0$; then it follows from $[\g,\g] = \g$ that
\[\g = [\nf,\nf] + [\nf,\Derive^i\rf] + [\Derive^i\rf,\Derive^i\rf] \subset \nf + \Derive^{i+1}\rf,\]
because $\nf$ is an ideal in $\g$.
Since $\rf$ is solvable, we may take $i$ such that $\Derive^i\rf=0$ to conclude that $\nf=\g$.
\end{proof}

\begin{lemma}[Perfect abelian extension of a semi-simple group]
\label{abelianextension}
Let $G$ be a perfect Lie group with Lie algebra $\g$.
If the radical $\rf$ of $\g$ is abelian, then the adjoint representation of $G$ is of class $\cP$.
\end{lemma}
\begin{proof}[Proof of Lemma~\ref{abelianextension}]
We have an exact sequence of $G$-modules
\[ 0 \to \rf \to \g \to \g/\rf \to 0,\]
and by Proposition~\ref{pr:subquo}\ref{it:subquo}, all we need to check is that both $\rf$ and $\g/\rf$ belong to $\cP(G)$.
Let $R$ be the solvable radical of $G$; it is equal to the closed connected subgroup of $G$ with Lie algebra $\rf$.
The Lie group $G/R$ is semi-simple, so its adjoint representation belongs to $\cP(G/R)$.
By Proposition~\ref{pr:subquo}\ref{it:asGmodHrep} , $\g/\rf$ is of class $\cP$ as a representation of $G$.

On the other hand, $\rf$ is totally reducible under the action of the semisimple group $S=G/R$, and moreover,
\[ \rf = [\s,\rf],\]
because $\g$ is perfect and $\rf$ abelian.
This implies that $\rf$ is a representation of class $\cP$ for $S$, and therefore for $G$ by Proposition~\ref{pr:subquo}~\ref{it:asGmodHrep}.
\end{proof}

\begin{remark}
If $G$ is not perfect, then $\g/[\g,\g]$ is non-zero, and $G$ acts trivially on $\g/[\g,\g]$, so that the adjoint representation does not belong to $\cP(G)$.
\end{remark}

\begin{remark}
It is not true in general that the adjoint representation of a perfect connected Lie group is of class $\cP$. Indeed, there exist perfect Lie algebras with non-trivial centers.
For instance, let $\cF_{2,2}$ denote the free 2-nilpotent Lie algebra over 2 generators $\mathsf{x},\mathsf{y}$. It is the Lie algebra of the Heisenberg group $H_3(\R)$. The action of $\SL(2,\R)$ on $\cF_{2,2}$ by linear substitution integrates to an action of $\SL(2,\R)$ on $H_3(\R)$ by group automorphisms. This allows us the define the Lie group $G = \SL(2,\R) \ltimes H_3(\R)$. Its Lie algebra $\g=\mathfrak{sl}(2,\R)\ltimes\cF_{2,2}$ is perfect. However, the adjoint representation of $G$ is not of class $\cP$, because $G$ acts trivially on the center of $\g$, generated by $[\mathsf{x},\mathsf{y}]$.
\end{remark}

\subsection{Abelian extensions of semi-simple groups}

Here, we prove Theorem~\ref{producttheorem} in the case where the Lie algebra of $G$ can be written as a semi-direct product $\g=\s\ltimes\rf$, with $\rf$ abelian. We shall see in \ref{subsec:rad} that the general case follows from this one.

We fix a connected perfect Lie group $G$ with Lie algebra $\g=\s\ltimes\rf$, where $\s$ is semi-simple and $\rf$ is an abelian ideal.
To prove Theorem~\ref{producttheorem} in this case, the idea is to apply Theorem~\ref{spclassp} to the adjoint representation of $G$ on its Lie algebra, and then to use the Campbell-Hausdorff formula.
Before that, we note that condition \ref{it:richinSfactors} in Theorem~\ref{producttheorem} automatically implies non-concentration for the image of $A$ under any non-trivial group homomorphism.

\begin{lemma}\label{lm:phiArich}
Let $G$ be a perfect connected Lie group.
Given a non-trivial homomorphism $\phi \colon G \to H$ to some connected Lie group $H$, there exists a neighborhood $U$ of the identity in $G$ such that the following holds.
Let $\eps > 0$ and $\kappa > 0$ be parameters and let $A \subset U$ be a subset satisfying condition~\ref{it:richinSfactors} of Theorem~\ref{producttheorem}.
Then
\[\forall \rho \geq \delta, \quad \NR{\phi(A)} \gg_{\phi} \delta^\eps\rho^{-\kappa}.\]
\end{lemma}
\begin{proof}
The isomorphism $G/\ker \phi \to \phi(G)$ is bi-Lipschitz when restricted to compact neighborhoods. Hence without loss of generality, we may assume that $H = G/\ker \phi$. Since $\ker \phi$ is closed, there exists a neighborhood $U$ of the identity in $G$, such that $\forall x, y$, $d(x^{-1}y,\ker\phi) = d(x^{-1}y,(\ker\phi)^{\circ})$. This allows us to further assume that $\ker \phi$ is connected.

Let $\nf$ be a maximal proper ideal of $\g$ containing the Lie algebra of $\ker\phi$.
By Lemma~\ref{lm:algFrattini}, $\nf$ is exactly the kernel of the projection of $\g$ to one of its simple factors.
It follows $\nf$ is the Lie algebra of a proper closed normal subgroup $N\lhd G$, with $G/N$ one of the simple factors of $G$.
We deduce the desired estimate from condition~\ref{it:richinSfactors} of Theorem~\ref{producttheorem} by using the fact that $G/\ker \phi \to G/N$ is $1$-Lipschitz.
\end{proof}

\begin{proof}[Proof of Theorem~\ref{producttheorem}, in the case where $\rf$ is abelian]
In this proof, implied constants in Landau and Vinogradov notations depend on $G$ and on the parameter $\kappa$.

By Lemma~\ref{abelianextension}, the adjoint representation of $G$ on $\g$ is of class $\cP$.
Setting $X = \log (A^{-1}\!A\cap B_G(1,\delta^{\eps})) \subset \g$, the hypotheses of Theorem~\ref{spclassp} are all met with $\eps$ replaced by $O(\eps)$.
Indeed, assumption~\ref{it:piArich} is guaranteed by Lemma~\ref{lm:phiArich}, and $A$ being a $\delta^\eps$-away from subgroups is exactly assumption~\ref{it:AawayH} of Theorem~\ref{producttheorem}.
So it remains to check that $X$ is $\delta^{O(\eps)}$-away from any proper submodule $W$ in $\g$.
We may assume that $W$ is maximal.
Then, it is a maximal proper ideal of $\g$, which by Lemma~\ref{lm:algFrattini} is equal to the kernel $H_i$ of some projection $\pi_i:\g\to \g/\h_i$ of $G$ on a simple factor.
In particular, there are only finitely many such $W$.
Shrinking the neighborhood $U$ if necessary, it suffices to check that $A^{-1}\!A\cap B_{G}(1,\delta^{\eps})$ is $\delta^{O(\eps)}$-away from $H_i$.
By assumption~\ref{it:richinSfactors}, for any $\rho \geq \delta$,
\begin{align*}
N(\pi_i(A^{-1}\!A\cap B_G(1,\delta^{\eps})),\rho)& \geq \max_g N(\pi_i(A\cap B_G(g,\delta^{\eps})),\rho)\\
& \geq \delta^{O(\eps)}N(\pi_i(A),\rho)\\
& \geq \delta^{O(\eps)}\rho^{-\kappa}.
\end{align*}
The last quantity is larger than $1$ if we choose $\rho = \delta^{C\eps}$ with a large constant $C = O(1)$. This shows that $A^{-1}\!A\cap B_G(1,\delta^{\eps})$ is $\delta^{O(\eps)}$-away from $\ker \pi_i$.

Thus, we can apply Theorem~\ref{spclassp} to get an integer $s \geq 1$ such that
\begin{equation}\label{ballinlie}
B_\g(0,\delta^{\eps_0}) \subset \bracket{A,X}_s+B_{\g}(0,\delta)
\end{equation}
when $\eps$ is small enough.

The idea is now to apply the Campbell-Hausdorff formula at an order $\ell$ such that the error term is of size at most $\delta$. We identify an element $w$ of the free group $F_s$ generated by $s$ elements and the word map $w:G^{\times s} \to G$ it induces.
If $x,y$ are elements in $\g$, we want to approximate $e^{x+y}$ by a word in $e^x, e^y$.
For example, with a remainder term of order 2, $e^{x+y}=e^xe^ye^{O(\|x\|^2+\|y\|^2)}$.
In order to get a remainder term of order 3, it is easier to approximate $e^{2(x+y)}$, and then, we get $e^{2(x+y)}=(e^x)^2(e^y)^2(e^y)^2e^x(e^y)^{-2}(e^x)^{-1}e^{O(\|x\|^3+\|y\|^3)}.$
We shall use the following lemma, which generalizes these elementary computations, and follows from the Campbell-Hausdorff formula.

\begin{lemma}
\label{lm:BCHnormed}
Let $\exp \colon \g \to G$ denote the exponential map of a Lie group. We fix a Euclidean norm on $\g$ and endow $G$ with the associated left-invariant Riemannian metric.
For all integers $s \geq 1$ and $\ell \geq 1$, there exists an integer $C \geq 1$, a word map $w \in F_s$ and a neighborhood $U$ of $0$ in $\g$ such that for all $x_1, \dotsc, x_s \in U$,
\[d\bigl(\exp(Cx_1 + \dotsb + Cx_s), w(\exp x_1,\dotsc,\exp x_s)\bigr) \ll_\ell (\norm{x_1} + \dotsb + \norm{x_s})^\ell.\]
\end{lemma}

\begin{proof}
Consider $\g$-valued functions $f$ defined on a neighborhood of $0$ in $\g^{\times s}$ that can be written as a sum of a convergent series
\[f(x_1,\dotsc,x_s) = \sum_{k = 1}^{+\infty} f_k(x_1,\dotsc,x_s)\]
where for each $k$, $f_k(x_1,\dotsc,x_s)$ is a $\Q$-linear combination of repeated brackets $[x_{i_1}, \dotsc, x_{i_k}] = [x_{i_1},[x_{i_2},\dotsc,[x_{i_{\ell-1}},x_{i_k}]\dots]]$ of length $k$. The series converges on $B_\g(0,r)^{\times s}$ for some $r > 0$ in the sense that the numerical series obtained by replacing each repeated bracket of length $k$ by $r^k$ and each coefficient by its absolute value is convergent. Identifying two such functions if they agree on a neighborhood of $0$, we get a linear space $\cG_s$ over $\Q$. Equipped with its obvious Lie bracket, $\cG_s$ is a graded Lie algebra over $\Q$. For $\ell \geq 1$, we write $O(d^\circ \geq \ell)$ to denote an unspecified element in $\cG_s$ of valuation at least $\ell$.

By the Baker-Campbell-Hausdorff formula \cite{dynkin}, the map defined by $(x,y) \mapsto x * y = \log(\exp(x)\exp(y))$ belongs to $\cG_2$ and moreover.
\begin{equation}\label{eq:BCH}
x * y = x + y + \frac{1}{2}[x,y] + O(d^\circ \geq 3).
\end{equation}
From that we deduce, by induction on $s$, that
\begin{equation}\label{eq:BCHorder1}
x_1 * \dotsb * x_s = x_1 + \dotsb + x_s + O(d^\circ \geq 2).
\end{equation}
We denote by $[x,y]_*$ the group commutator $x * y * (-x) * (-y)$ and by $[x_1,\dotsc,x_s]_*$ the repeated group commutator $[x_1, [x_2 \dotsc,[x_{s-1},x_s]_*\dots]_*]_*$. We have by \eqref{eq:BCH},
\[[x,y]_* = [x,y] + O(d^\circ \geq 3)\]
and again by induction on $s$,
\begin{equation}\label{eq:*commutators}
[x_1,\dotsc,x_s]_* = [x_1,\dotsc,x_s] + O(d^\circ \geq s+1).
\end{equation}
Now we prove by induction on $\ell$ that there exists an integer $C_\ell$ and a word $w_\ell \in F_s$ such that
\begin{equation}\label{eq:BCHorderl}
x_1 + \dotsb + x_s = w_\ell^*\bigl(\frac{x_1}{C_\ell},\dotsc,\frac{x_s}{C_\ell}\bigr) + O(d^\circ \geq \ell),
\end{equation}
where $w_\ell^*$ is the word map induced by $w_\ell$, which is well defined on a neighborhood of $0$ in $\g^{\times s}$.
For $\ell = 2$, this is given by \eqref{eq:BCHorder1}. Suppose the result has been proved for some $\ell\geq 2$.
Let $f$ be the sums of terms of degree $\ell$ in the remainder term $O(d^\circ\geq\ell)$ on the right-hand side of \eqref{eq:BCHorderl}.
Since $f$ has rational coefficients, there is an integer $C \geq 1$ such that we can write
\[ f(x_1,\dotsc ,x_s) = \sum_{i = 1}^N m_i\bigl(\frac{x_1}{C},\dotsc,\frac{x_s}{C}\bigr)\]
where each $m_i$ is a repeated bracket of length $\ell$.
Therefore, by \eqref{eq:*commutators} and \eqref{eq:BCHorder1}, there is $w' \in F_s$ a product of repeated commutators such that
\begin{equation*}
f(x_1,\dotsc ,x_s) = w'^*\bigl(\frac{x_1}{C},\dotsc,\frac{x_s}{C}\bigr) + O(d^\circ \geq \ell + 1).
\end{equation*}
Thus,
\begin{align*}
x_1 + \dotsb + x_s &= w_\ell^*\bigl(\frac{x_1}{C_\ell},\dotsc,\frac{x_s}{C_\ell}\bigr) + w'^*\bigl(\frac{x_1}{C},\dotsc,\frac{x_s}{C}\bigr) + O(d^\circ \geq \ell + 1)\\
&= w_\ell^*\bigl(\frac{x_1}{C_\ell},\dotsc,\frac{x_s}{C_\ell}\bigr)*w'^*\bigl(\frac{x_1}{C},\dotsc,\frac{x_s}{C}\bigr) + O(d^\circ \geq \ell + 1).
\end{align*}
In the last step we used the fact that $w'^*\bigl(\frac{x_1}{C},\dotsc,\frac{x_s}{C}\bigr)$ has valuation at least $\ell$. This finishes the proof of the induction step and concludes the proof of the lemma.
\end{proof}

To conclude the proof of Theorem~\ref{producttheorem} in the case $\rf$ is abelian, we choose $\ell > \frac{1}{\eps}$ and apply Lemma~\ref{lm:BCHnormed} to elements $x_i$ of the form $x_i = \Ad(a_i)y_i$, with $a_i \in A^s$ and $y_i \in X$.
By definition $X \subset B_{\g}(0,\delta^\eps)$, so the error term is indeed of size $O_s(\delta^{\ell\eps}) = O(\delta)$, and therefore,
\begin{align*}
\exp[C\Ad(a_1)y_1+ \dotsb + C\Ad(a_s)y_s] &\in w(a_1e^{y_1}a_1^{-1},\dots,a_se^{y_s}a_s^{-1})B_G(1,O(\delta))\\
&\in (A \cup \{1\} \cup A^{-1})^{s'}B_G(1,O(\delta)),
\end{align*}
for some $s' = O_{s,\ell}(1)$. Recalling (\ref{ballinlie}), we obtain
\begin{align*}
B_G(1,\delta^{\eps_0})
& \subset \exp[C \cdot B_{\g}(0,\delta^{\eps_0})]\\
& \subset \exp[C \cdot \bracket{A,X}_s +B_{\g}(0,C\delta))]\\
& \subset A^{s'} B_G(1, O(\delta)).
\end{align*}
This finishes the proof of the theorem in the case $\rf$ is abelian.
\end{proof}

\subsection{Proof of the product theorem, general case}
\label{subsec:rad}

We now explain how to deal with a perfect Lie group $G$ with Lie algebra $\g=\s\ltimes\rf$, where $\rf$ is nilpotent by Lemma~\ref{lm:radnil} but not abelian.
This will follow from the previous case, together with a quantitative version of the following fact: If $R$ is a nilpotent Lie group, a subset $A\subset R$ generates the group $R$ if and only if $A\mod [R,R]$ generates $R/[R,R]$.

For $A$ and $B$ subsets of a group $G$, we shall write $[A,B]$ to denote the set of all commutators $[a,b]$, $a \in A$, $b \in B$.
This notation is in conflict with the group theoretic commutator which is the subgroup generated by all commutators.
Despite this inconvenience, it will be clear from the context what $[A,B]$ means.
The precise lemma that we shall use is as follows.

\begin{lemma}\label{lm:BRl1}
Let $R$ be a connected nilpotent Lie group with descending central series $R_i$, $i \geq 1$, i.e. $R_1 = R$ and for $i\geq 1$, $R_{i+1} = [R,R_i]$.
For each $i\geq 1$ there is $k \geq 1$ such that for all $\rho > 0$ small enough,
\[
B_{R_{i+1}}(1,\rho^2) \subset [B_R(1,\rho),B_{R_i}(1,\rho)]^k.
\]
\end{lemma}
\begin{proof}
Denote by $\rf_i$, $i \geq 1$ the descending central series of the Lie algebra $\rf$. Let $(z_1,\dotsc,z_m)$ be a basis of $\rf_{i+1}$ consisting of commutators $z_j = [x_j,y_j]$ with $x_j \in \rf$ and $y_j \in \rf_i$. For each $j$, consider the map $f_j \colon \R \to R_{i+1}$ defined by
\[f_j(t) =
\begin{cases}
  [\exp(\sqrt{t}x_j),\exp(\sqrt{t}y_j)] & \quad \text{if } t \geq 0 \\
  [\exp(\sqrt{-t}y_j),\exp(\sqrt{-t}x_j)]& \quad \text{if } t < 0
\end{cases}
\]
and further define $f \colon \R^m \to R_{i+1}$ by $f(t_1,\dotsc,t_m) = f_1(t_1) \dotsm f_m(t_m)$.
The function $f$ is of class $C^1$ and its differential at $0$ is
\[T_0f(h_1,\dotsc,h_m) = h_1z_1 + \dotsb + h_mz_m,\]
so it is a $C^1$-diffeomorphism on a neighborhood of $0$.
This implies that for some constant $c>0$ depending only on $R$,
\[ B_{R_{i+1}}(1,c\rho^2) \subset f(B_\R(0,\rho)) \subset [B_R(1,\rho),B_{R_i}(1,\rho)]^m.\]
This finishes the proof of the lemma, because for $\rho$ small enough, $B_{R_{i+1}}(1,c\rho^2)\cdot B_{R_{i+1}}(1,c\rho^2) \supset B_{R_{i+1}}(1,2c\rho^2)$.
\end{proof}

We are now ready to finish the proof of Theorem~\ref{producttheorem}.

\begin{proof}[Proof of Theorem~\ref{producttheorem}, general case]
Here again implied constants in Landau and Vinogradov notations depend on $G$ and $\kappa$.
Write the Lie algebra of $G$ as a semi-direct product $\g=\s\ltimes\rf$, with $\s$ semi-simple and $\rf$ a nilpotent ideal, and let $R$ be the nilpotent radical of $G$, i.e. the closed connected normal subgroup of $G$ with Lie algebra $\rf$.
The proof goes by induction on the nilpotency class $\ell$ of $R$.

We have already seen that Theorem~\ref{producttheorem} holds if $\ell \leq 1$.
Now suppose that $R$ has nilpotency class equal to $\ell\geq 2$ and that Theorem~\ref{producttheorem} has been proved if the nilpotency class is strictly less than $\ell$.

Let $R_i$, $i \geq 1$ denote the lower central series of the group $R$.
Each $R_i$ , $i \geq 1$ is closed and connected, and the Lie algebra of $R_i$ is exactly the $i$-th term in the lower central series of $\rf$, see e.g. \cite[Theorem 5.7, p. 55]{OnishchikVinberg}.
We first remark that the assumptions of Theorem~\ref{producttheorem} are preserved when projecting to a quotient.
The nilpotency class of the radical of $G/R_\ell$ is $\ell-1$.
Let $\eps_1>0$ be some constant, whose value will be specified later.
By the induction hypothesis applied to $G/R_\ell$, provided $\eps$ is small enough compared to $\eps_1$, for some integer $s$ depending on $\kappa$ and $\eps_1$,
\[B_{G}(1,\delta^{\eps_1}) \subset (A\cup \{1\} \cup A^{-1})^s B_G(1,\delta) R_\ell\]
Without loss of generality, we may replace $(A\cup \{1\} \cup A^{-1})^s B_G(1,\delta)$ by $A$, and assume that
\[B_{R}(1,\delta^{\eps_1}) \subset (R\cap A) R_\ell \quad \text{and} \quad B_{R_{\ell-1}}(1,\delta^{\eps_1}) \subset (R_{\ell-1} \cap A)R_\ell.\]
By Lemma~\ref{lm:BRl1}, we also have
\[B_{R_\ell}(1,\delta^{2\eps_1}) \subset [B_{R}(1,\delta^{\eps_1}),B_{R_{\ell-1}}(1,\delta^{\eps_1})]^{O(1)}.\]
From these inclusions and the fact that $R_\ell$ is in the center of $R$, it follows that
\begin{equation}\label{2eps1}
B_{R_\ell}(1,\delta^{2\eps_1}) \subset A^{O(1)}B_G(1,O(\delta)).
\end{equation}
At this stage replace $A^{O(1)}B_G(1,O(\delta))$ by $A$. The fact that $B_{R}(1,\delta^{\eps_1}) \subset A R_\ell$ and $B_{R_\ell}(1,\delta^{2\eps_1}) \subset A$ does not prove what we want yet but gives the lower bound
\[ N(A^2, \delta) \gg_G \delta^{-\dim(G) + O(\eps_1)}.\]
Covering $A^2$ by balls of radius $\frac{1}{2}\delta^{3\eps_1}$, we obtain
\[ N(A^{-2}\!A^2 \cap B_G(1,\delta^{3\eps_1}), \delta) \gg_G \delta^{-\dim(G) + O(\eps_1)}.\]
Write $A' = A^{-2}\!A^2 \cap B_G(1,\delta^{3\eps_1})$. Then $A'$ satisfies the assumptions of Theorem~\ref{producttheorem} with $\kappa = 1$ and $\eps = O(\eps_1)$. Hence if $\eps_1$ is small enough compared to $\eps_0$, then by the induction hypothesis again,
\[B_G(1,\delta^{\eps_0}) \subset A'^s B_G(1,\delta) R_\ell\]
for some $s$ depending on $\eps_0$. Since any element in $R_\ell$ involved in this inclusion is within distance $\delta^{2\eps_1}$ from the identity, we can conclude using \eqref{2eps1} that
\[B_G(1,\delta^{\eps_0}) \subset A'^{O(1)} B_G(1,\delta) A.\]
This finishes the proof of Theorem~\ref{producttheorem}.
\end{proof}

\subsection{Approximate subgroups in non-perfect Lie groups}
Here we prove Proposition~\ref{perfectisoptimal}. First, observe that in a nontrivial abelian Lie group, generalized arithmetic progressions (i.e. sums of arithmetic progressions) are the prototypes of approximate subgroups. Then in a non-perfect Lie group $G$, it suffices to lift a generalized arithmetic progression in its abelianization $G/[G,G]$ to obtain an approximate subgroup with the desired properties.

\begin{proof}[Proof of Proposition~\ref{perfectisoptimal}]
First consider the abelian case $G = \R^{\times d}$, with $d \geq 1$. Let $\kappa \in {(0,1]}$. Given a neighborhood $U$ of $0 \in \R^{\times d}$, let $r > 0$ be such that $B_{\R^d}(0,r) \subset U$. Define
\[P = \bigl\{ \delta^\kappa x \in \R^{\times d} \mid x \in \Z^{\times d} \cap {[-\delta^{-\kappa}r,\delta^{-\kappa}r]}^{\times d}\bigr\}.\]
It is easy to check that $P$ satisfy the required properties.

Now let $G$ be a simply connected non-perfect Lie group. Then $G/[G,G] \simeq \R^{\times d}$ where $d = \dim \g - \dim [\g,\g]$. Let $\pi \colon G \to \R^{\times d}$ the projection.  Given a neighborhood $U$ of $1_G \in G$, let $r > 0$ be such that $B_{G}(1_G,2r) \subset U$ and $B_{\R^d}(0,r) \subset \pi(U)$. Let $P$ be defined as above and put $A = B_G(1_G,2r) \cap \pi^{-1}(P)$.

On the one hand,
\[\ND{A} \approx_{G,r} \delta^{-\dim [\g,\g]} \ND{\pi(A)} \approx_{G,r}\delta^{-\dim [\g,\g]} \ND{P} \approx_{G,r} \delta^{-\dim[\g,\g] - d\kappa},\]
and for similar reason,
\[\ND{AAA} \ll_{G,r} \delta^{-\dim [\g,\g]} \ND{P+P+P} \ll_{G,r} \ND{A}.\]

On the other hand, when $\delta$ is small so that $\delta^\kappa < r$, $A$ is $\delta^\kappa$-dense in $B_G(1_G,r)$, that is, 
\[B_G(1_G,r) \subset A^{(\delta^\kappa)}.\] 
It follows immediately that for any connected normal subgroup $N \lhd G$, $\pi_{G/N}(A)$ is $\delta^\kappa$-dense in $B_{G/N}(1_{G/N},r)$ and hence
\[\forall \rho \geq \delta, \NR{\pi_{G/N}(A)} \gg_{G,r} \rho^{-\kappa}.\]
Moreover, it is not difficult to see that given a simply connected Lie group $G$ and $r > 0 $, there is $c = c(G,r) > 0$ such that no proper closed connected subgroup is $c$-dense in $B_G(1_G,r)$. From this we deduce that $A$ is $(c - \delta^\kappa)$-away from proper closed connected subgroups.
\end{proof}

\bibliographystyle{plain}
\bibliography{bibliography}

\end{document}